\documentclass[leqno,12pt]{article}
\usepackage{amsmath,amssymb,amsthm,mathtools,a4wide}
\usepackage{microtype,booktabs,array,needspace}
\usepackage{color}
\usepackage{tikz}
\usetikzlibrary{arrows.meta}
\usepackage[backend=biber,style=numeric,isbn=false,maxnames=15]{biblatex}
\DefineBibliographyStrings{english}{pages = {}, page = {}}
\usepackage{xurl}
\usepackage{hyperref}
\hypersetup{
  hidelinks,
  pdftitle={Linear crossing families via dual levels},
  pdfauthor={Tim Gehrunger}
}

\newcommand{\Bigskip}{\bigskip\bigskip}
\let\le\leqslant
\let\ge\geqslant

\let\phi\varphi
\let\theta\vartheta
\let\epsilon\varepsilon
\let\setminus\smallsetminus

\newtheorem{theorem}{Theorem}[section]
\newtheorem{proposition}[theorem]{Proposition}
\newtheorem{lemma}[theorem]{Lemma}
\newtheorem{corollary}[theorem]{Corollary}
\theoremstyle{remark}
\newtheorem*{remark}{Remark}
\makeatletter
\let\c@equation\c@theorem
\makeatother

\makeatletter
\renewenvironment{proof}[1][\proofname]{\par
  \pushQED{\qed}%
  \normalfont \topsep6\p@\@plus6\p@\relax
  \trivlist
  \item[\hskip\labelsep\bfseries #1\@addpunct{.}]\ignorespaces
}{%
  \popQED\endtrivlist\@endpefalse
}
\makeatother

\DeclareMathOperator{\conv}{conv}
\newcommand{\cR}{\mathcal R}
\newcommand{\cB}{\mathcal B}
\newcommand{\cF}{\mathcal F}
\newcommand{\eps}{\varepsilon}

\title{\strut
  \vskip-80pt
  Linear crossing families via dual levels
}
\author{
  \begin{minipage}[t]{.43\textwidth}
    \centering
    Tim Gehrunger\\[12pt]
    \small Department of Mathematics \\
    ETH Z\"urich\\
    8092 Z\"urich\\
    Switzerland \\
    tim.gehrunger@math.ethz.ch
  \end{minipage}
}
\date{\today}

\begin{document}
\maketitle
\Bigskip

\begin{abstract}
Building on the work of Pach, Rubin, and Tardos, we prove that every set of $n\ge2$ points in the real plane with no three collinear contains at least $cn$ pairwise crossing segments with distinct endpoints, for an absolute constant $c>0$.

\end{abstract}

\section{Introduction}\label{sec:intro}
How many pairwise crossing segments must a planar point set in general position determine? Here, \emph{general position} means that no three points are collinear. A \emph{crossing family} consists of segments joining points of the set, with all endpoints distinct, any two of which intersect in their relative interiors. Since each segment uses two points, an $n$-point set has no crossing family larger than $\lfloor n/2\rfloor$. Aronov et al.~\cite{AEGK91,AEGK} introduced crossing families in 1991 and conjectured that every such set contains a crossing family of size proportional to $n$; see also~\cite[Problem~1]{ABV}. They also  proved the lower bound $\Omega(\sqrt n)$. Later, Pach, Rubin, and Tardos~\cite[Theorem~1.1(i)]{PRT} proved the improved bound $n/2^{O(\sqrt{\log n})}$. In this article, we remove the subpolynomial loss and establish the conjectured linear order of growth:

\begin{theorem}
\label{thm:main}
There is an absolute constant $c>0$ such that every set of $n\ge2$ points in the plane in general position contains a crossing family of size at least $cn$.
\end{theorem}

\paragraph{AI use and formalization.}
The proof of Theorem \ref{thm:main} was obtained using an updated version of ProofCouncil \cite{ProofCouncil} during preparations for the third batch of First Proof \cite{FirstProof}. The author checked and revised the system’s output, which is available in \cite{InitialVersion}. Theorem 1.1 has been formalized in Lean using Codex \cite{LeanCode}. The formalization contains an unconditional proof of Theorem \ref{thm:main}, but does not include the applications to plane trees and spoke sets. Codex was also used to assist with drafting and revising this manuscript. The ProofCouncil run used GPT-6 Astra as its author agent and the recorded model-API cost  was \$32.36. These costs were covered by Swiss National Science Foundation grant 10009122, “Beyond Benchmark Scores: Analyzing AI Reasoning on Research-Level Mathematics”.

\paragraph{Overview.}
We deduce Theorem~\ref{thm:main} from a decomposition result for two separated point sets. Here $A,B$ are \emph{separated} if their convex hulls $\conv(A),\conv(B)$ are disjoint, and a \emph{joining segment} has one endpoint in each set. Let $A,B$ be separated $m$-point sets, where $m\ge1$, whose union is in general position. Their \emph{avoidance defect} is
\begin{equation}\label{eq:defect}
\begin{split}
t(A,B)={}&\bigl|\{\{a,a'\}\in\tbinom A2:
          \ell(a,a')\cap\conv(B)\ne\varnothing\}\bigr|\\
       &+\bigl|\{\{b,b'\}\in\tbinom B2:
          \ell(b,b')\cap\conv(A)\ne\varnothing\}\bigr|.
\end{split}
\end{equation}
Here $\ell(p,q)$ is the supporting line through $p,q$, and $\binom X2$ denotes the unordered two-element subsets of $X$.

\begin{proposition}
\label{prop:deletion}
Let $m\ge1$ be an integer, and let $A,B$ be separated $m$-point sets in the plane whose union is in general position. After deleting at most $2t(A,B)$ of the $m^2$ segments joining $A$ to $B$, the remaining segments can be partitioned into $2m-1$ possibly empty crossing families. In particular, one such family has size at least
\begin{equation}\label{eq:bound}
\frac{m^2-2t(A,B)}{2m-1}.
\end{equation}
If $t(A,B)\le m^2/4$, there is a crossing family of size at least $m/4$.
\end{proposition}

Pach, Rubin, and Tardos~\cite[Lemma~3.3]{PRT} proved that every sufficiently large general-position set of $n$ points contains separated subsets $A,B$ of common size $m$ proportional to $n$, with $t(A,B)\le m^2/4$. Applying Proposition~\ref{prop:deletion} to these sets, we obtain a crossing family of size at least $m/4$. 
%Thus a fixed bound on the proportion of defective pairs suffices. 
Our contribution is to obtain a linear-size crossing family from a fixed near-avoidance bound.

To prove Proposition \ref{prop:deletion}, we replace each point $(r,s)$ by the line $y=rx-s$ in a second plane, coloring lines from $A$ red and lines from $B$ blue. Each joining segment is represented by the intersection of its endpoints' two lines. We group these intersections according to how many lines lie below them, obtaining $2m-1$ levels. Each pair counted by the avoidance defect also corresponds to an intersection: one of two same-color lines, with opposite-color lines both above and below it.  Thus each defect belongs to exactly one level.

During a left-to-right sweep at level $k$, repeated appearances of a line would give segments with a common endpoint. Between two such appearances, however, the line must enter or leave the set of the $k$ lowest lines at a defective same-color intersection. We use these transitions to mark lines and discard all their red--blue intersections at that level. Using the same transition argument, we show that the surviving intersections have the relative order required for pairwise crossing. We bound the number discarded by twice the number of defects at the level. Since each defect belongs to exactly one level, at most $2t(A,B)$ segments are deleted.

\paragraph{Structure of the paper.}
{In Section~\ref{sec:extraction}, we deduce Theorem~\ref{thm:main} from Proposition~\ref{prop:deletion} using the result of Pach, Rubin, and Tardos, and record applications to plane trees and spoke sets. We set up duality and the defect correspondence in Section~\ref{sec:dual}, bound deletions in Section~\ref{sec:level}, and prove the crossing property in Section~\ref{sec:crossing}, completing the proof of Proposition~\ref{prop:deletion}.}

%\newpage
\section{From near-avoidance to a linear bound}\label{sec:extraction}
We use the following form of \cite[Lemma~3.3]{PRT}, expressed in terms of the avoidance defect.
\begin{lemma}
\label{lem:prt-extraction}
There is an absolute constant $K_0>0$ such that, for every positive integer $m$ and every $1>\eps>0$ and $\delta>0$, a geometric graph $G=(V,E)$ whose vertices are in general position and which satisfies
\[
|V|\ge\frac{K_0m}{\eps^4\delta^5},
\qquad |E|\ge\delta|V|^2
\]
contains separated subsets $A,B\subseteq V$ with
\[
|A|=|B|=m,\qquad t(A,B)\le\eps m^2,
\]
and at least $\delta m^2$ edges joining $A$ to $B$.
\end{lemma}

In \cite{PRT}, the sets $A$ and $B$ form an $\eps$-avoiding pair: the total number of unordered incomparable pairs in two associated partial orders is at most $\eps m^2$~\cite[Definition~3.1]{PRT}. A pair is incomparable precisely when its supporting line meets the convex hull of the opposite set \cite[Lemma~2.2]{PRT}. This count is therefore exactly $t(A,B)$.

\begin{corollary}
\label{thm:extraction}
There is an absolute constant $K\ge2$ such that, for every $0<\eps<1$, every positive integer $m$, and every general-position point set $P$ of cardinality
\[
n\ge K\eps^{-4}m
\]
in the real plane, there are separated subsets $A,B\subseteq P$ satisfying
\[
|A|=|B|=m,\qquad t(A,B)\le\eps m^2.
\]
\end{corollary}
\begin{proof}
 Apply Lemma \ref{lem:prt-extraction} to the complete geometric graph on $P$, with $\delta=1/4$ and
\[
K=\max\{2,4^5K_0\}.
\]
Under the hypothesis, we have $n\ge2$, so this graph has $\binom n2\ge n^2/4$ edges. With this choice of $K$, the size condition also holds, and we obtain separated $m$-point sets with $t(A,B)\le\eps m^2$.
\end{proof}

\begin{proof}[Proof of Theorem~\ref{thm:main}]
Fix $\epsilon=1/4$, take $K$ from Corollary~\ref{thm:extraction}, choose an integer
\[
H\ge K4^4,
\]
and set $c=1/(8H)>0$. Since $K\ge2$, we have $H\ge2\cdot4^4\ge2$.

Let $P$ be a general-position set of $n\ge2$ points in the plane. First suppose that $n\ge2H$, and put $m=\lfloor n/H\rfloor\ge2$. Then
\[
n\ge Hm\ge K4^4m,\qquad
m\ge\frac nH-1\ge\frac n{2H}.
\]
Applying Corollary~\ref{thm:extraction} for $\eps=1/4$, we get separated subsets $A,B\subseteq P$ satisfying
\[
|A|=|B|=m,\qquad t(A,B)\le\frac{m^2}{4}.
\]
By Proposition~\ref{prop:deletion}, there is a crossing family of size at least
\[
\frac m4\ge\frac n{8H}=cn.
\]

If $2\le n<2H$, choose any two points of $P$. Their segment is a crossing family of size one, and
\[
1\ge\frac n{8H}=cn.
\]
\end{proof}

For a point set $P$, write $K(P)$ for the complete geometric graph on $P$. A tree is \emph{plane} if no two of its edges cross. In an edge partition into plane trees, each edge belongs to exactly one tree. Note that here the trees may share vertices and need not be spanning. Bose, Hurtado, Rivera-Campo, and Wood~\cite[Problem~16]{BHRW} asked whether every complete geometric graph on $n$ vertices admits a partition into at most $(1-\eps)n$ plane subgraphs for some absolute constant $\eps>0$. We can now answer this question affirmatively, with trees as the parts.

\begin{corollary}
\label{cor:plane-tree-partition}
There is an absolute constant $\eps>0$ such that, for every set $P$ of $n\ge2$ points in the plane in general position, the edges of $K(P)$ can be partitioned into at most $(1-\eps)n$ plane trees.
\end{corollary}
\begin{proof}
Bose et al.~\cite[Lemma~14]{BHRW} proved that a complete geometric graph on $n$ vertices containing a crossing family of size $k$ admits an edge partition into $n-k$ plane trees. By Theorem~\ref{thm:main}, we can choose such a family with $k\ge cn$. The required partition therefore has at most $n-k\le(1-c)n$ trees, and we take $\eps=c$.
\end{proof}

We next consider \emph{packings} of plane spanning trees. Here every tree contains all vertices, and no edge belongs to two trees, but some edges of the complete graph may remain unused. Biniaz and Garc\'ia~\cite{BG} proved that every complete geometric graph on $n$ vertices in general position contains at least $\lfloor n/3\rfloor$ edge-disjoint plane spanning trees. We obtain a linear bound while requiring each tree to have diameter at most three. Here the diameter of a tree is the maximum number of edges on a path.

\begin{corollary}
\label{cor:spanning-double-stars}
There is an absolute constant $c>0$ such that, for every set $P$ of $n\ge2$ points in the plane in general position, the graph $K(P)$ contains at least $cn$ edge-disjoint plane spanning trees, each of diameter at most three.
\end{corollary}
\begin{proof}
In the proof of~\cite[Theorem~1]{AHK}, Aichholzer et al. construct $k$ edge-disjoint plane spanning trees from a crossing family of size $k$. Each tree is a \emph{double star}: it consists of one edge of the crossing family together with an edge from every other vertex to one of its endpoints. Thus every tree has diameter at most three. We apply this construction to a crossing family of size $k\ge cn$, chosen by Theorem~\ref{thm:main}.
\end{proof}

\begin{remark}
A \emph{spoke set} for $P$ is a set of pairwise nonparallel lines such that every open unbounded region of their arrangement contains a point of $P$. By slightly rotating the supporting lines of a crossing family of size $k$ about the segment midpoints, one obtains a spoke set of size $k$; see~\cite{Schnider}. Applying this construction to a crossing family from Theorem~\ref{thm:main}, we obtain a spoke set of size at least $cn$ for every general-position set of $n\ge2$ points. This answers the question of Bose et al.~\cite[Problem~17]{BHRW} affirmatively.
\end{remark}

\section{Defects in the dual arrangement}\label{sec:dual}
To prove Proposition~\ref{prop:deletion}, fix separated $m$-point sets $A,B$, where $m\ge1$, whose union is in general position.

We represent a point $p=(p_x,p_y)$ by the line
\[
p^*(x)=p_xx-p_y,
\]
identifying the line with its height as a function of $x$. Lines representing points of $A$ are colored red, and those representing points of $B$ blue. Once the coordinates are chosen so that the lines have distinct slopes, each joining segment is represented by the intersection of its two lines.

Since $A$ and $B$ are separated, we can rotate and translate the configuration so that every point of $A$ has negative first coordinate and every point of $B$ has positive first coordinate. These first coordinates are the slopes of the corresponding dual lines. Hence every red slope is negative and every blue slope is positive. This will make all red--blue intersections behave in the same way when we move from left to right.

For vectors $u=(u_x,u_y)$ and $v=(v_x,v_y)$, write $\det(u,v)=u_xv_y-u_yv_x$. The sign of $\det(q-p,z-p)$ records on which side of the directed line from $p$ to $q$ the point $z$ lies: positive is left and negative is right.

\Needspace{8\baselineskip}
\begin{lemma}
\label{lem:side-tests}
Let $X$ be a finite point set in general position.
\begin{enumerate}
\item For distinct $p,q\in X$ and a nonempty set $Y\subseteq X\setminus\{p,q\}$, the line $\ell(p,q)$ meets $\conv(Y)$ if and only if $Y$ has points on both sides of the line.
\item Two segments with four distinct endpoints in $X$ cross in their relative interiors if and only if the endpoints of each segment lie on opposite sides of the supporting line of the other.
\end{enumerate}
\end{lemma}
\begin{proof}
No point of $Y$ lies on $\ell(p,q)$. If $Y$ lies in one open half-plane, so does its convex hull; otherwise a segment joining two points of $Y$ on opposite sides meets the line. This proves the first assertion.

For the second, each side condition places the unique intersection of the supporting lines in the relative interior of one segment, so together they give a proper crossing. Conversely, a segment crossing a line in its relative interior has endpoints on opposite sides, since neither endpoint lies on the line.
\end{proof}

Both tests depend only on the orientation signs of triples. We can therefore perturb the points to examine dual intersections one at a time, provided these signs and separation are preserved. We say that a line is \emph{spanned} by a point set if it contains two of its points.

\begin{samepage}
\begin{lemma}
\label{lem:generic-position}
After rotating and translating the configuration, we can perturb the labelled points arbitrarily slightly so that every point of $A$ has negative first coordinate, every point of $B$ has positive first coordinate, all first coordinates are distinct, and no two distinct spanned lines are parallel. Applying the perturbation preserves separation, all triple-orientation signs, the avoidance defect $t(A,B)$, and every proper crossing relation between segments with distinct endpoint labels.
\end{lemma}
\end{samepage}
\begin{proof}
Rotate and translate a strict separating line to the vertical axis, with $A$ on its left and $B$ on its right. These strict coordinate inequalities and all triple-orientation signs persist in sufficiently small, pairwise disjoint open disks around the labelled points.

Within these disks, choose the new points successively. At each choice, avoid the vertical lines through previously chosen points, the lines through pairs of previously chosen points, and the lines through a previously chosen point parallel to a line through two previously chosen points. There are only finitely many forbidden lines, so they cannot cover the open disk. The new spanned lines are not parallel to any earlier spanned line, nor to one another, since collinearity with earlier points is excluded.

The points remain separated by the vertical axis. By Lemma~\ref{lem:side-tests}, their unchanged orientation signs preserve the defect and all proper crossing relations. Keeping endpoint labels thus transfers crossing families back with the same cardinalities.
\end{proof}

We work with the perturbed sets, retaining the notation $A,B$, and write $\cR,\cB$ for their red and blue line families. Distinct first coordinates give distinct dual slopes, and general position ensures that no three dual lines pass through the same point. If $p^*\cap q^*=(s,h)$, then the supporting line $\ell(p,q)$ has equation $y=sx-h$. In particular, the horizontal coordinate $s$ is its slope. Since distinct spanned lines are not parallel, all dual intersections have distinct horizontal coordinates.

We will also use the relation between sides in the two planes. For any point $z$, its signed vertical difference from the line $y=sx-h$ is
\begin{equation}\label{eq:side}
z_y-(sz_x-h)=-\bigl(z^*(s)-h\bigr).
\end{equation}
Thus a point above the original supporting line corresponds to a dual line below $(s,h)$, and a point below it corresponds to a dual line above.

Under the bijection described above, the $m^2$ joining segments correspond to the red--blue intersections.

\medskip
Next, we group the intersections into $2m-1$ classes. An intersection $q=(s,h)$ has \emph{level $k$} if exactly $k-1$ lines $\lambda$ satisfy $\lambda(s)<h$; that is, their heights at horizontal coordinate $s$ are strictly less than $h$. There are $2m$ lines in total, of which exactly two pass through the intersection, so $1\le k\le2m-1$.

To study one level, we move a vertical line from left to right through the arrangement. Away from intersections, let $L_k(x)$ be the set of the $k$ lowest lines at horizontal coordinate $x$. We track which lines belong to this set. At an intersection with horizontal coordinate $s$, write $L_k(s^-)$ and $L_k(s^+)$ for its membership immediately before and after the intersection. The coordinate preparation ensures that only one intersection is passed at a time. We record the possible changes in the following lemma.

\begin{lemma}
\label{lem:sweep-rule}
At an intersection of level $k$, the two lines through it exchange ranks $k$ and $k+1$: one leaves $L_k$, and the other enters. Intersections at other levels leave $L_k$ unchanged. At a red--blue level-$k$ intersection, the red line enters $L_k$ and the blue line leaves it.
\end{lemma}
\begin{proof}
There are $k-1$ lines below a level-$k$ intersection and no third line through it. The two lines through this intersection therefore occupy adjacent ranks $k$ and $k+1$ immediately before and after the event, and their ranks are exchanged. An intersection at another level swaps two lines that are either both in $L_k$ or both outside it, so membership in $L_k$ does not change.

At a red--blue level-$k$ intersection, the red line has negative slope and the blue line has positive slope. As we pass the intersection from left to right, the red line therefore moves from above the blue line to below it. Consequently, the red line enters $L_k$ and the blue line leaves it.
\end{proof}

We say that a red--red intersection is \emph{defective} if some blue line lies strictly above it and some blue line lies strictly below it. Similarly, define defective blue--blue intersections by interchanging the colors. Let $T_k$ be the number of defective same-color intersections of level $k$.

\begin{lemma}
\label{lem:defect-correspondence}
For distinct $a_1,a_2\in A$, the red--red intersection $a_1^*\cap a_2^*$ is defective if and only if
\[
\ell(a_1,a_2)\cap\conv(B)\ne\varnothing.
\]
The analogous statement holds for pairs in $B$. Each defective same-color intersection arises from exactly one such unordered pair and we have
\begin{equation}\label{eq:sumT}
\sum_{k=1}^{2m-1}T_k=t(A,B).
\end{equation}
\end{lemma}
\begin{proof}
Let $(s,h)=a_1^*\cap a_2^*$. By Lemma~\ref{lem:side-tests}, the line $\ell(a_1,a_2)$ meets $\conv(B)$ exactly when $B$ has points on both sides of it. By~\eqref{eq:side}, this happens exactly when blue lines lie both above and below $(s,h)$. Thus the pair $\{a_1,a_2\}$ contributes to $t(A,B)$ precisely when its red--red intersection is defective. The same reasoning applies to pairs in $B$ and blue--blue intersections.

Distinct unordered pairs give distinct intersections, and each intersection has exactly one level. Counting the defective intersections over all levels therefore counts every pair contributing to $t(A,B)$ exactly once, giving~\eqref{eq:sumT}.
\end{proof}

\section{Deletion at each level}\label{sec:level}
Now fix a level $k$. We will use its $T_k$ defective intersections to discard at most $2T_k$ joining segments. Throughout the argument, levels and $L_k(x)$ are computed in the original arrangement of all $2m$ lines and are not recomputed after deletions.

\paragraph{The marking rule.}
By Lemma~\ref{lem:sweep-rule}, a red line enters $L_k$ at each of its red--blue intersections of level $k$. To participate a second time, it must first leave $L_k$. Similarly, a blue line leaves $L_k$ at each red--blue intersection of level $k$, so it must enter $L_k$ again before the next such intersection. These exits and entrances motivate the following rule, applied at intersections of level $k$:
\begin{itemize}
\item At a defective red--red intersection, we mark the red line that leaves $L_k$.
\item At a defective blue--blue intersection, we mark the blue line that enters $L_k$.
\item At all other intersections, we assign no marks.
\end{itemize}
Each defective intersection at level $k$ gives one mark, and a line may receive several marks. A line is called \emph{marked} at this level if it receives at least one.

Marks are assigned separately at each level, and all marks are assigned before any intersections are retained. We say that a red--blue level-$k$ intersection is \emph{clean} if neither of the two lines through it is marked at this level. Note that here a mark affects every red--blue intersection on its line at that level, including earlier ones. We discard precisely the non-clean intersections and their corresponding joining segments. We will show that the retained segments have distinct endpoints and cross pairwise.
%Discarding every appearance of a marked line ensures that both lines at each retained intersection are unmarked. This condition will rule out both repeated endpoints and the transitions that would prevent two surviving segments from crossing.

When $t(A,B)=0$, there are no defects and hence no marks: the rule keeps every joining segment.

\begin{lemma}
\label{lem:marked-transition}
Suppose a red line $r$ lies strictly between two blue lines throughout an open interval of horizontal coordinates. Every exit of $r$ from $L_k$ within that interval gives a mark to $r$. Similarly, if a blue line $b$ lies strictly between two red lines throughout such an interval, every entrance of $b$ into $L_k$ within the interval gives a mark to $b$. In particular, an unmarked line cannot make the corresponding transition.
\end{lemma}
\begin{proof}
At a red--blue level-$k$ intersection, the red line enters $L_k$ and the blue line leaves it, by Lemma~\ref{lem:sweep-rule}. Consequently, a red line can leave $L_k$ only at a red--red intersection of level $k$. The two blue lines lie on opposite sides of that intersection, making it defective, and the rule marks the exiting red line. Likewise, a blue entrance must occur at a blue--blue intersection of level $k$. The two red lines make it defective, so the entering blue line receives a mark.
\end{proof}

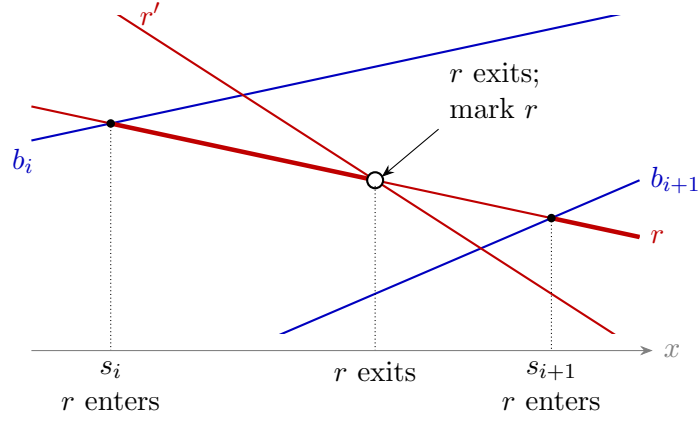
\begin{figure}[tb]
\centering
\begin{tikzpicture}[x=3.5cm,y=.75cm,>=Stealth,
  every node/.style={font=\small}]
\begin{scope}
\clip (-1.3,-2.7) rectangle (1,2.9);
\draw[blue!75!black,thick] (-1.3,.7)--(1,3);
\draw[blue!75!black,thick] (-1.3,-4.6)--(1,0);
\draw[red!75!black,thick] (-1.3,3.9)--(1,-3);
\draw[red!75!black,thick] (-1.3,1.3)--(1,-1);
\draw[red!75!black,line width=1.7pt] (-1,1)--(0,0);
\draw[red!75!black,line width=1.7pt] (.666667,-.666667)--(1,-1);
\end{scope}
\node[blue!75!black,below left] at (-1.25,.75) {$b_i$};
\node[blue!75!black,right] at (1,0) {$b_{i+1}$};
\node[red!75!black,right] at (1,-1) {$r$};
\node[red!75!black,above] at (-.85,2.55) {$r'$};
\fill (-1,1) circle (1.6pt);
\fill (.666667,-.666667) circle (1.6pt);
\draw[fill=white,line width=.8pt] (0,0) circle (3pt);
\draw[->] (.24,.9)--(.025,.06);
\node[anchor=south west,align=left] at (.24,.9) {$r$ exits;\\mark $r$};
\draw[densely dotted] (-1,.95)--(-1,-3);
\draw[densely dotted] (0,-.1)--(0,-3);
\draw[densely dotted] (.666667,-.72)--(.666667,-3);
\draw[->,gray] (-1.3,-3)--(1.05,-3) node[right] {$x$};
\node[below,align=center] at (-1,-3) {$s_i$\\$r$ enters};
\node[below] at (0,-3) {$r$ exits};
\node[below,align=center] at (.666667,-3) {$s_{i+1}$\\$r$ enters};
\end{tikzpicture}
\caption{Two consecutive red--blue intersections of $r$ at level $k$ (here $k=2$), with horizontal coordinates $s_i$ and $s_{i+1}$. The bold portions indicate where $r$ belongs to $L_k(x)$, the set of the $k$ lowest lines. Between its two entrances, the line $r$ exits at a red--red intersection. At that exit, the blue line $b_i$ lies above $r$ and the blue line $b_{i+1}$ lies below it, so the intersection is defective and the marking rule marks $r$.}
\label{fig:marking}
\end{figure}

\paragraph{Repeated appearances and the number of deletions.}
For a line $\lambda\in\cR\cup\cB$, let $d_k(\lambda)$ be the number of red--blue level-$k$ intersections on it, and let $u_k(\lambda)$ be the number of marks it receives at this level. We next show that every appearance after the first requires a distinct mark; see Figure~\ref{fig:marking} for a sketch of the red-line case.

\begin{lemma}
\label{lem:repeat}
With the fixed arrangement and level-$k$ marking rule above, every line $\lambda\in\cR\cup\cB$ satisfies
\[
d_k(\lambda)\le1+u_k(\lambda).
\]
\end{lemma}
\begin{proof}
Fix a red line $r$, and let $s_1<\cdots<s_d$ be the horizontal coordinates of its red--blue level-$k$ intersections, with blue partners $b_1,\ldots,b_d$. There is nothing to prove if $d\le1$. Between its entrances at $s_i$ and $s_{i+1}$, the line $r$ must leave $L_k$. Moreover, since red slopes are negative and blue slopes are positive, we have
\[
b_i(x)>r(x)>b_{i+1}(x)\qquad(s_i<x<s_{i+1}).
\]
By Lemma~\ref{lem:marked-transition}, the line $r$ receives a mark at an exit in each of these disjoint intervals, so $u_k(r)\ge d-1$.

A blue line $b$ must enter $L_k$ between any two consecutive exits. Its red partners satisfy
\[
r_i(x)<b(x)<r_{i+1}(x),
\]
throughout the interval between their intersections with $b$. By Lemma~\ref{lem:marked-transition}, the line $b$ receives a mark at an entrance in each of these disjoint intervals, proving the assertion for both colors.
\end{proof}

There are two consequences. If $\lambda$ is unmarked, then $u_k(\lambda)=0$ and $d_k(\lambda)\le1$. Thus no line occurs in two retained intersections at this level, and the corresponding segments have distinct endpoints. If $\lambda$ is marked, then $u_k(\lambda)\ge1$, so
\[
d_k(\lambda)\le1+u_k(\lambda)\le2u_k(\lambda).
\]
%This second inequality bounds the number of intersections discarded because of a marked line in terms of its number of marks.

\begin{lemma}
\label{lem:count}
At most $2T_k$ red--blue level-$k$ intersections are not clean.
\end{lemma}
\begin{proof}
Let $S=\{\lambda\in\cR\cup\cB:u_k(\lambda)>0\}$ be the set of marked lines. Each defective intersection gives one mark, so the total number of marks is $T_k$:
\[
\sum_{\lambda\in S}u_k(\lambda)=T_k.
\]
Every discarded intersection lies on at least one marked line. Using the inequality just established, the number discarded is therefore at most
\[
\sum_{\lambda\in S}d_k(\lambda)
\le2\sum_{\lambda\in S}u_k(\lambda)=2T_k.
\]
An intersection on two marked lines may be counted twice, which is harmless for an upper bound.
\end{proof}

\section{From clean intersections to crossing families}\label{sec:crossing}
It remains to show that the segments corresponding to clean intersections at level $k$ cross pairwise. By Lemma~\ref{lem:side-tests} and the duality relation~\eqref{eq:side}, this amounts to showing that, at either intersection, the two lines through the other intersection lie on opposite sides of it.

\begin{lemma}
\label{lem:order}
Let $q_1=(s_1,h_1)=r_1\cap b_1$ and $q_2=(s_2,h_2)=r_2\cap b_2$ be distinct clean red--blue intersections at level $k$, where $s_1<s_2$, the lines $r_1,r_2$ are red, and $b_1,b_2$ are blue. At the earlier intersection $q_1$, the lines through $q_2$ satisfy
\begin{equation}\label{eq:earlier}
r_2(s_1)>h_1>b_2(s_1).
\end{equation}
At the later intersection $q_2$, the lines through $q_1$ satisfy
\begin{equation}\label{eq:later}
b_1(s_2)>h_2>r_1(s_2).
\end{equation}
\end{lemma}

\begin{proof}
An unmarked line has $u_k=0$ and hence passes through at most one red--blue level-$k$ intersection by Lemma~\ref{lem:repeat}. Since all four lines through these two intersections are unmarked, they are distinct. No third line passes through either intersection, so none of the comparisons in the statement can be an equality. We prove the four strict comparisons in turn. In each case, the opposite inequality would require an unmarked line to make a transition that is impossible by Lemma~\ref{lem:marked-transition}.

\smallskip\noindent\emph{At the earlier intersection.}
Suppose that $r_2(s_1)<h_1$. Then $r_2\in L_k(s_1^+)$, whereas $r_2\notin L_k(s_2^-)$ because $r_2$ enters at $s_2$. Thus $r_2$ must leave $L_k$ somewhere in $(s_1,s_2)$. Throughout this interval, we have
\[
b_1(x)>r_2(x)>b_2(x).
\]
The first inequality follows from $r_2(s_1)<b_1(s_1)$ and the negative slope of $r_2-b_1$; the second holds before the intersection of $r_2$ and $b_2$ at $s_2$. By Lemma~\ref{lem:marked-transition}, the exit gives a mark to $r_2$. This contradicts the fact that $q_2$ is clean, proving $r_2(s_1)>h_1$.

Suppose that $b_2(s_1)>h_1$. Then $b_2\notin L_k(s_1^+)$, whereas $b_2\in L_k(s_2^-)$ because $b_2$ leaves at $s_2$. Thus the line $b_2$ must enter $L_k$ somewhere in $(s_1,s_2)$. Throughout $(s_1,s_2)$, we have
\[
r_1(x)<b_2(x)<r_2(x).
\]
The first inequality follows from $b_2(s_1)>r_1(s_1)$ and the positive slope of $b_2-r_1$; the second holds before the intersection at $s_2$. By Lemma~\ref{lem:marked-transition}, the line $b_2$ receives a mark at the entrance, again contradicting the fact that $q_2$ is clean. Hence $b_2(s_1)<h_1$, completing the two comparisons in~\eqref{eq:earlier}.

\smallskip\noindent\emph{At the later intersection.}
Suppose that $r_1(s_2)>h_2$. Then $r_1\in L_k(s_1^+)$ because it enters at $s_1$, but $r_1\notin L_k(s_2^-)$. Thus the line $r_1$ must leave $L_k$ somewhere in $(s_1,s_2)$. Throughout $(s_1,s_2)$, we have
\[
b_1(x)>r_1(x)>b_2(x).
\]
The first inequality holds after the intersection at $s_1$. For the second inequality, note that $r_1(s_2)>b_2(s_2)$ and that the slope of $r_1-b_2$ is negative, so
$$r_1(x)-b_2(x)>r_1(s_2)-b_2(s_2)>0$$
for $x<s_2$. By Lemma~\ref{lem:marked-transition}, the line $r_1$ receives a mark at the exit, contradicting the fact that $q_1$ is clean. Hence $r_1(s_2)<h_2$.

Finally, suppose that $b_1(s_2)<h_2$. Then $b_1\notin L_k(s_1^+)$ because it leaves at $s_1$, but $b_1\in L_k(s_2^-)$. Thus the line $b_1$ must enter $L_k$ somewhere in $(s_1,s_2)$. Throughout $(s_1,s_2)$, we have
\[
r_1(x)<b_1(x)<r_2(x).
\]
The first inequality holds after the intersection at $s_1$. For the second inequality, note that $b_1(s_2)<r_2(s_2)$ and that the slope of $b_1-r_2$ is positive, so
$$b_1(x)-r_2(x)<b_1(s_2)-r_2(s_2)<0$$
for $x<s_2$. By Lemma~\ref{lem:marked-transition}, the line $b_1$ receives a mark at the entrance, again contradicting the fact that $q_1$ is clean. Thus $b_1(s_2)>h_2$, completing the two comparisons in~\eqref{eq:later}.
\end{proof}

\begin{corollary}
\label{lem:crossing}
The joining segments corresponding to the clean red--blue intersections at level $k$ have distinct endpoints and cross pairwise in their interiors.
\end{corollary}
\begin{proof}
Since both lines through each clean intersection are unmarked, the corresponding segments have distinct endpoints by Lemma~\ref{lem:repeat}. Choose any two clean intersections and use the notation of Lemma~\ref{lem:order}. Let $a_1,a_2\in A$ be dual to $r_1,r_2$ and let $c_1,c_2\in B$ be dual to $b_1,b_2$. We must show that $[a_1,c_1]$ and $[a_2,c_2]$ cross in their relative interiors.

The supporting line $\ell(a_1,c_1)$ is $y=s_1 x-h_1$. By~\eqref{eq:side}, a point in the original plane lies above this line exactly when its dual line has height less than $h_1$ at $s_1$. By~\eqref{eq:earlier}, the points $a_2$ and $c_2$ therefore lie strictly on opposite sides of $\ell(a_1,c_1)$. Applying the same identity at $(s_2,h_2)$ and using~\eqref{eq:later}, we find that $a_1$ and $c_1$ lie strictly on opposite sides of $\ell(a_2,c_2)$.

Both side conditions in Lemma~\ref{lem:side-tests} hold, so the two segments cross in their relative interiors. This applies to every pair of retained segments at the level.
\end{proof}

\begin{proof}[Proof of Proposition~\ref{prop:deletion}]
Note that it suffices to work in the perturbed configuration of Lemma~\ref{lem:generic-position}. For $1\le k\le2m-1$, let $\cF_k$ consist of the joining segments corresponding to clean red--blue intersections of level $k$. By Corollary~\ref{lem:crossing}, each $\cF_k$ is a crossing family. These families are pairwise disjoint, since every red--blue intersection has a unique level.

By Lemma~\ref{lem:count} and~\eqref{eq:sumT}, we have
\[
\sum_{k=1}^{2m-1}|\cF_k|
\ge m^2-2\sum_{k=1}^{2m-1}T_k=m^2-2t(A,B).
\]
Thus at most $2t(A,B)$ joining segments are deleted, and averaging proves~\eqref{eq:bound}. If $t(A,B)\le m^2/4$, then
\[
\frac{m^2-2t(A,B)}{2m-1}\ge\frac{m^2}{2(2m-1)}\ge\frac m4.
\]
Finally, we transfer the families to the original configuration by their endpoint labels, as in Lemma~\ref{lem:generic-position}, preserving their cardinalities, crossing relations, and partition of the retained segments.
\end{proof}

\paragraph{Acknowledgments.}
Most importantly, I want to thank Raphael Steiner for suggesting that ProofCouncil should be run on this problem and for his comments on an earlier version of this article. I am also very grateful to the ProofCouncil team, specifically Johannes Schmitt, Jasper Dekoninck, David Holmes and Yonggang Jiang for being involved with the continued development of ProofCouncil. I also thank OpenAI for providing a free ChatGPT account through the ChatGPT for Academic Researchers program. 

\printbibliography
\end{document}